\documentclass[12pt,reqno]{amsart}
\usepackage[a4paper, top=1.5cm, bottom=1.5cm, left=2cm, right=2cm, heightrounded, marginparwidth=3.5cm, marginparsep=.2cm, centering]{geometry}
\usepackage[utf8]{inputenc}
\usepackage[english]{babel}
\usepackage{amsmath}
\usepackage{mathrsfs}
\usepackage{amsfonts}
\usepackage{amsthm}
\usepackage{amssymb}
\usepackage{latexsym}
\usepackage{stmaryrd}
\usepackage[shortlabels]{enumitem}
\usepackage{tikz-cd}
\usepackage{fullpage}
\usepackage{pdfpages}
\usepackage[colorlinks=true,hyperindex,pagebackref,linktocpage=true]{hyperref}
\usepackage[nameinlink]{cleveref}
\hypersetup{
	colorlinks,
	linkcolor={red!40!black},
	citecolor={cyan!70!black},
	urlcolor={orange!60!black}
}
\newtheorem{thm}{Theorem}[section]
\newtheorem*{thm2}{Theorem}

\newtheorem{prop}[thm]{Proposition}

\newtheorem{cor}[thm]{Corollary}
\newtheorem{lem}[thm]{Lemma}
\theoremstyle{definition}
\newtheorem{defi}[thm]{Definition}
\newtheorem{rmk}[thm]{Remark}

\newcommand{\bb}[1]{\mathbb{#1}}
\newcommand{\cl}[1]{{\mathcal{#1}}}
\newcommand{\mfr}[1]{{\mathfrak{#1}}}
\newcommand{\mrm}[1]{{\mathrm{#1}}}

\newcommand{\perf}{{\mathrm{perf}}}

\newcommand{\red}{{\mathrm{red}}}
\newcommand{\Frob}{{\mathrm{Frob}}}
\newcommand{\Spec}{{\mathrm{Spec}}}

\newcommand{\Spf}{{\mathrm{Spf}}}
\newcommand{\Sppf}{{\mathrm{Sppf}}}

\newcommand{\GR}{\mathrm{Gr}}

\newcommand{\Lie}{\mathrm{Lie}}

\newcommand{\deff}{{\mathrm{def}}}

\newcommand{\ov}[1]{{\overline{#1}}}
\newcommand{\wtd}[1]{{\widetilde{#1}}}

\newcommand{\invp}{{[\tfrac{1}{p}]}}
\newcommand{\invpp}{{[1/p]}}

\newcommand{\cp}[1]{{{#1}^{\wedge}}}
\newcommand{\hsp}[1]{\hspace{#1}}

\title{Equidimensionality of affine Deligne-Lusztig varieties in mixed characteristic}
\author{Yuta Takaya}
\address{Graduate School of Mathematical Sciences, The University of Tokyo, 3-8-1 Komaba,
Meguro-ku, Tokyo 153-8914, Japan}
\email{takaya@ms.u-tokyo.ac.jp}
\begin{document}
\begin{abstract}
We prove the equidimensionality of affine Deligne-Lusztig varieties in mixed characteristic.
This verifies a conjecture made by Rapoport and implies that the results of Nie and Zhou-Zhu can be extended to the whole irreducible components of affine Deligne-Lusztig varieties. 
The method is to translate the work of Hartl-Viehmann into mixed characteristic and construct local foliations for affine Deligne-Lusztig varieties. 
This leads us to develop a theory of formal algebraic geometry for perfect schemes. 
\end{abstract}

\maketitle
\tableofcontents

\section*{Introduction}

In \cite{Rap05}, affine Deligne-Lusztig varieties were introduced to describe the underlying spaces of moduli spaces of $p$-divisible groups or local shtukas with level structures. 
The geometry of affine Deligne-Lusztig varieties depends on level structures. 
Let $G$ be a connected reductive group over a non-archimedean local field $F$. 
Let $O$ be the ring of integers of $F$. 
In this paper, we fix a reductive model $\cl{G}$ of $G$ over $O$ and study affine Deligne-Lusztig varieties at the hyperspecial level $\cl{G}$. 

Let $L$ be the completion of the maximal unramified extension of $F$ and fix a Borel pair $T\subset B\subset G$. Let $O_L$ be the ring of integers of $L$ and let $\sigma$ be the Frobenius automorphism of $L$ relative to $F$. 
Let $k$ be the residue field of $O$ and $\ov{k}$ be that of $O_L$. 
Associated with an element $b\in G(L)$ and a dominant cocharacter $\mu\in X_*(T)_+$, the affine Deligne-Lusztig variety $X_\mu(b)$ is defined to be a $\ov{k}$-scheme with the following set of closed points:
\[X_\mu(b)(\ov{k})=\Bigl\{g\in G(L)/\cl{G}(O_L) \Bigm\vert g^{-1}b\sigma(g) \in \cl{G}(O_L) \pi^{\mu} \cl{G}(O_L)\Bigr \}.\]
The affine Deligne-Lusztig variety $X_\mu(b)$ is locally of finite type over $\ov{k}$ in equal characteristic, and locally perfectly of finite type in mixed characteristic. 
Its geometric properties have been studied from various perspectives. 
For example, its non-emptiness criterion was proved in \cite{RR96} and \cite{Gas10} and its dimension formula was derived in \cite{GHKR06}, \cite{Vie06}, \cite{Ham15} and \cite{Zhu17}.
Recently, a representation-theoretic description of the $J_b(F)$-orbits of top-dimensional irreducible components of $X_\mu(b)$ was obtained. 
It was first conjectured by Chen and Zhu and proved in \cite{ZZ20} and \cite{Nie22}.
A bijection between the $J_b(F)$-orbits and Mirkovi\'{c}-Vilonen cycles was constructed in \cite{Nie22}. 
The stabilizer of each top-dimensional irreducible component was also studied in \cite{HZZ21}. 

Though these previous works only concern top-dimensional irreducible components,
it has been expected that the affine Deligne-Lusztig variety $X_\mu(b)$ is actually equidimensional. 
In the literature, this expectation first appeared in \cite[Conjecture 5.10]{Rap05}. 
It is based on the analogous property of Newton strata of the reduction of Shimura varieties discovered in \cite{Oor01}. 
This problem was recently addressed again in \cite[\S 3]{HV18} and \cite[Remark 0.1]{Nie22}. 
 
In equal characteristic, the equidimensionality of $X_\mu(b)$ was already proved by Hartl and Viehmann. 
The basic case was first treated in \cite{HV11} and the split case was carried out in their subsequent paper \cite{HV12}. 
Though not fully completed in the literature, the argument of the latter paper works in general as stated in \cite[Theorem 2.9]{VW18}. 
However, $X_\mu(b)$ was not given a perfect scheme structure before the pioneering work of \cite{Zhu17} and \cite{BS17} in mixed characteristic, and its equidimensionality remains open. 
When $X_\mu(b)$ can be identified with the perfection of the underlying space of a Rapoport-Zink space, it can be studied concretely via $p$-divisible groups and its equidimensionality can be proven in some cases as in \cite[Theorem 3.4]{HV18}. 
However, this approach can be applied only when $\mu$ is minuscule and the pair $(G,\mu)$ is reasonable, e.g. of local Hodge type. 

In this paper, we prove equidimensionality in mixed characteristic in full generality. 

\begin{thm2}\textup{(\Cref{thm:mainthm})}
The closed affine Deligne-Lusztig variety $X_{\leq \mu}(b)$ is equidimensional. 
\end{thm2}

This is the exact counterpart of \cite[Corollary 6.8 (a)]{HV12} in mixed characteristic. Here, we note that $X_{\leq \mu}(b)$ is a perfect scheme locally perfectly of finite type over $\ov{k}$ and $X_{\mu}(b)$ is an open subscheme of $X_{\leq \mu}(b)$. In particular, the above theorem implies that $X_{\mu}(b)$ is also equidimensional. Moreover, combining with the dimension formula of $X_\mu(b)$, we see that $X_\mu(b)$ is dense in $X_{\leq\mu}(b)$. Thus, the description of the $J_b(F)$-orbits of top-dimensional irreducible components of $X_{\mu}(b)$ can be enhanced to that of irreducible components of $X_{\leq \mu}(b)$. 

Our proof is based on the method in \cite{HV12}, so let us review their proof in short. 
Take a closed point $[g]$ of $X_{\leq \mu}(b)$ and fix a representative $g\in G(L)$. 
Set $b'=g^{-1}b\sigma(g)$ and $\mu^*=w_0(-\mu)$ with $w_0$ the longest element of the Weyl group. 
Over the completion of the affine Schubert variety $\mrm{Gr}_{\cl{G},\leq \mu^*}$ at $[b'^{-1}]$, we can construct the universal deformation of the local $\cl{G}$-shtuka $(\cl{G},b'\sigma)$. 
The closed Newton stratum $N^b$ of the universal deformation space is, up to a finite surjective map, decomposed into the product of the completion $X_{\leq \mu}(b)^{\wedge}_{[g]}$ of $X_{\leq \mu}(b)$ at $[g]$ and a central leaf $I^\wedge_{x,n}$. 
This is a local analogue of Oort's foliations for $p$-divisible groups and abelian varieties constructed in \cite{Oor04}. 
This decomposition allows us to relate the dimension of $X_{\leq \mu}(b)$ at $[g]$ to the dimension of $N^b$, which can be estimated through the purity of the Newton stratification. 

Our strategy is to work out this method with $F$-crystals with $\cl{G}$-structure in the category of perfect $\ov{k}$-schemes. This leads us to develop a theory of formal algebraic geometry for perfect schemes. Roughly speaking, when we define the completion of a perfect scheme $X$ along a closed perfect subscheme $Z$, we take a deperfection $Z_0\hookrightarrow X_0$ of $Z\hookrightarrow X$ and regard the perfection of the completion $(X_0)^\wedge_{Z_0}$ as the completion of $X$ along $Z$. We construct a deformation of an $F$-crystal with $\cl{G}$-structure $(\cl{G},b'\sigma)$ and take the closed Newton stratum $N^b$ of the deformation space. Then, we construct the following correspondence between $N^b$ and $X_{\leq \mu}(b)^{\wedge}_{[g]}\mathrel{\widehat{\times}} I^\wedge_{x,n}$. 
\begin{center}
    \[
    \begin{tikzcd}
        &(\wtd{N}^b)^{\wedge}\arrow[ld]\arrow[rd]& \\
        X_{\leq \mu}(b)^\wedge_{[g]}\mathrel{\widehat{\times}} I_{x,n}^\wedge && N^b
    \end{tikzcd}
    \]
\end{center}
Here, $\wtd{N}^b\to N^b$ is the perfection of a finite surjective map and $(\wtd{N}^b)^{\wedge}$ is the completion of $\wtd{N}^b$. This correspondence can be regarded as a local foliation in our setting and constructed in a similar way as \cite{HV12}. 
The reason why we only have a correspondence is the lack of the universality of the deformation of $(\cl{G},b'\sigma)$. 
The key property of this correspondence is that though the map $(\wtd{N}^b)^\wedge \to X_{\leq \mu}(b)^{\wedge}_{[g]} \mathrel{\widehat{\times}}I^\wedge_{x,n}$ may not come from a finite surjective map, we can still show that it is adic. This property enables us to estimate the dimension of $X_{\leq \mu}(b)$ at $[g]$. 

This paper is organized as follows. 
In \Cref{sec:1}, we define the completion of perfect schemes locally perfectly of finite type over a perfect field and develop a dimension theory needed to handle the dimension of $(\wtd{N}^b)^{\wedge}$. 
In \Cref{sec:2}, we introduce the notion of continuity of sections of loop groups, positive loop groups and affine Grassmannians. We make use of this result to define $N^b$. 
In \Cref{sec:3}, we recall the notion of fundamental alcoves that is a key technical ingredient in \cite{HV12}. 
In \Cref{sec:4}, we construct a local foliation as a correspondence and use it to deduce the equidimensionality of $X_{\leq \mu}(b)$. 

\section*{Acknowledgements}
I would like to thank my advisor Yoichi Mieda for his constant support and encouragement. I am also grateful to Ryosuke Shimada for introducing this topic to me. This work was supported by the WINGS-FMSP program at the Graduate School of Mathematical Sciences, the University of Tokyo.

\section*{Notation}
We fix a prime number $p$. All rings are assumed to be commutative. 
The Frobenius of a ring of characteristic $p$ is denoted by $\Frob$. 
The perfection of a ring $A$ of characteristic $p$ is denoted by $A^\perf$. 
An adic ring is a topological ring equipped with a linear topology of a finitely generated ideal. 
In our terminology, adic rings are not necessarily Noetherian nor complete, and the completion of an adic ring $A$ is denoted by $A^\wedge$. 
Similarly, the completion of a ring $A$ with respect to the $I$-adic topology is denoted by $A^\wedge_I$. 

\section{Formal algebraic geometry for perfect schemes}\label{sec:1}

The aim of this section is to introduce the spectra of perfect adic rings perfectly formally of finite type over a perfect field $k$ and define the completion of perfect schemes locally of finite type over $k$. 
Moreover, we develop a dimension theory for local adic rings. We will use it to measure the dimension of the completion of the perfection of local Noetherian rings.  

\subsection{The spectra of perfect adic rings}\label{ssec:specpa}

\begin{defi}
    Let $A$ be an adic ring of characteristic $p$ with an ideal of definition $I$. 
    We call the perfection $A^\perf$ equipped with the $I$-adic topology the perfection of $A$. 
\end{defi}

\begin{defi}
    An adic ring is called a perfect adic ring if it is isomorphic to the perfection of a complete adic ring. 
    Such a complete adic ring is called a deperfection. 
\end{defi}

Note that we distinguish perfect adic rings and adic perfect rings in our terminology. Adic perfect rings are adic rings of characteristic $p$ that are perfect as rings of characteristic $p$. 

Perfect adic rings play the role of affine building blocks in the formal algebraic geometry of perfect schemes. 
We will introduce their spectra as the perfection of the formal spectra of their deperfections. Here, the perfection is taken in the following sense. 

\begin{defi}
    Let $A$ be a complete adic ring of characteristic $p$ with an ideal of definition $I$. 
    Let $\cl{O}^\perf_{\Spf(A)}$ be the perfection of the structure sheaf on $\Spf(A)$ endowed with the $I$-adic topology on each affine open subspace. 
    We call the topologically ringed space $(\lvert \Spf(A) \rvert, \cl{O}^\perf_{\Spf(A)})$ the perfection of $\Spf(A)$. 
\end{defi}

The fact that $\cl{O}^\perf_{\Spf(A)}$ is a sheaf of topological rings follows from the following lemma. 

\begin{lem}
    Let $A$ be a complete adic ring of characteristic $p$ with an ideal of definition $I$. 
    Let $B$ be an $I$-completely faithfully flat complete adic $A$-algebra. 
    The perfection induces a topological embedding $A^\perf\hookrightarrow B^\perf$. 
\end{lem}
\begin{proof}
    It is enough to see that the map $A^\perf/I\cdot A^\perf\to B^\perf/I\cdot B^\perf$ is injective. 
    It is a colimit of $A/\Frob^N(I)\cdot A\to B/\Frob^N(I)\cdot B$ over every positive integer $N$. 
    If we assume that $I$ is finitely generated, we see that $\Frob^N(I)\cdot A$ is an ideal of definition of $A$. 
    Thus, $A/\Frob^N(I)\cdot A\to B/\Frob^N(I)\cdot B$ is injective since $B$ is $I$-completely faithfully flat. 
\end{proof}

There is subtle dependency on a choice of a deperfection when we define a spectrum of a perfect adic ring. 

\begin{lem}\label{lem:indep1}
    Let $A$ and $B$ be reduced complete adic rings of characteristic $p$. 
    Let $f\colon A^\perf\to B^\perf$ be an isomorphism of their perfection. 
    If we have $f(A)\subset \Frob^{-N}(B)$ and $f^{-1}(B)\subset \Frob^{-N}(A)$ for some positive integer $N$, $f$ induces an isomorphism of the perfection of $\Spf(A)$ and that of $\Spf(B)$. 
\end{lem}
\begin{proof}
    It is easy to see that $f$ induces a homeomorphism $\lvert \Spf(A) \rvert \cong \lvert \Spf(B) \rvert $, so it is enough to check on affine open subsets. 
    Take an element $a\in A$ such that $f(a)\in B$. 
    It is enough to show that $f$ induces an isomorphism between $(A^\wedge_{(a)})^\perf$ and $(B^\wedge_{(f(a))})^\perf$. 
    The assumption on $f$ induces a map from $A^\wedge_{(a)}$ to $\Frob^{-N}(B^\wedge_{(f(a))})$ and one from $B^\wedge_{(f(a))}$ to $\Frob^{-N}(A^\wedge_{(a)})$. These two maps induce a desirable isomorphism. 
\end{proof}

To resolve this dependency, we impose a mild finiteness condition on perfect adic rings. 
Let $k$ be a perfect field of characteristic $p$. 

\begin{defi}
    A perfect adic $k$-algebra is said to be perfectly formally of finite type if it is isomorphic to the perfection of a complete adic $k$-algebra formally of finite type. 
    Such a complete adic $k$-algebra is called a deperfection formally of finite type. 
\end{defi}

\begin{lem}\label{lem:indep2}
    Let $A$ and $B$ be reduced complete adic $k$-algebras formally of finite type. 
    For any isomorphism $f\colon A^\perf\to B^\perf$ of adic $k$-algebras, there is a positive integer $N$ such that $f(A)\subset \Frob^{-N}(B)$ and $f^{-1}(B)\subset \Frob^{-N}(A)$. 
\end{lem}
\begin{proof}
    Let $A_0$ be a $k$-algebra of finite type with an ideal $I$ such that $A$ is isomorphic to $(A_0)^\wedge_I$. 
    Fix an isomorphism $\iota\colon A\cong (A_0)^\wedge_I$. 
    Since $A_0$ is of finite type over $k$, $(f\circ \iota^{-1})(A_0)$ lies in $\Frob^{-N}(B)$ for some positive integer $N$. 
    Then, we have $f(A)\subset \Frob^{-N}(B)$ from the continuity of $f$. 
    The opposite direction can be handled in the same way. 
\end{proof}

We can finally define the spectrum of a perfect adic ring perfectly formally of finite type over $k$. 
Note that the reduction of a complete Noetherian adic ring is automatically complete. 

\begin{defi}
    Let $A$ be a perfect adic ring perfectly formally of finite type over $k$. 
    Let $A_0$ be a reduced deperfection of $A$ formally of finite type over $k$. 
    The perfect formal spectrum of $A$ is defined as the perfection of $\Spf(A_0)$ and denoted by $\Sppf(A)$.  
    It is a topologically ringed space that is independent of the choice of $A_0$. 
\end{defi}

Note that we have natural morphisms $\Spf(A^\wedge)\to \Sppf(A)\to \Spec(A)$ of topologically ringed spaces. 
The first morphism is the right adjoint to the category of perfect formal schemes and the second one is the left adjoint to the category of perfect schemes. 
Here, perfect formal schemes mean formal schemes of characteristic $p$ whose structure sheaves are perfect. 
Since $A$ is the perfection of a Noetherian ring, the underlying space of $\Spec(A)$ is Noetherian. 

Distinguished open subspaces of $\Sppf(A)$ can be described in the usual way. 
For an element $f\in A$, the open subspace $D(f)$ is isomorphic to $\Sppf(A_0[\tfrac{1}{f}])^{\wedge,\perf}$, where $A_0$ is a deperfection of $A$ formally of finite type over $k$ containing $f$. 
We may introduce a class of topologically ringed spaces Zariski locally isomorphic to the spectrum of a perfect adic ring, which we call \textit{formal perfect schemes}.

\subsection{Completion of perfect schemes}\label{ssec:formcomp}

Let $k$ be a perfect field of characteristic $p$. 
Recall that a perfect $k$-scheme $X$ is said to be locally perfectly of finite type if it is Zariski locally isomorphic to the spectrum of the perfection of a finite type $k$-algebra. 
A $k$-scheme $X_0$ locally of finite type with an isomorphism $X_0^\perf \cong X$ is called a deperfection of $X$. 
Basically, when we define the completion of $X$ along a closed perfect subscheme $Z$, we take a deperfection $X_0$ of $X$ and a closed subscheme $Z_0$ of $X_0$ that is a deperfection of $Z$, and then use the perfection of the completion $(X_0)^\wedge_{Z_0}$. 
We will justify this construction in the following. 

Let $X$ be a perfect $k$-scheme perfectly locally of finite type and $Z$ be a closed perfect subscheme of $X$. 
First, suppose that $X=\Spec(A)$ is affine. 
Take a deperfection $A_0$ of $A$ of finite type over $k$.
Let $I_0$ be an ideal of $A_0$ such that $Z=\Spec(A_0/I_0)^\perf$. 
Then, we define the completion $X^\wedge_Z$ of $X$ along $Z$ as $\Sppf((A_0)^\wedge_{I_0})^\perf$. 
As we see in the following lemma, it is independent of the choice of $A_0$ and $I_0$. 
For an ideal $J$ of a ring $B$ of characteristic $p$, the ideal $\bigcup_{N\geq 0} \Frob^{-N}(J)$ of $B^\perf$ is denoted by $J^\perf$. 

\begin{lem}
    Let $A_0$ and $A_1$ be finite type $k$-algebras and let $f\colon A_0^\perf\cong A_1^\perf$ be an isomorphism. 
    For every ideal $I_0$ of $A_0$ and every ideal $I_1$ of $A_1$ such that $f(I_0^\perf)=I_1^\perf$, $f$ induces an isomorphism $((A_0)^\wedge_{I_0})^\perf\cong ((A_1)^\wedge_{I_1})^\perf$ of adic $k$-algebras. 
\end{lem}
\begin{proof}
    We may suppose that $A_0$ and $A_1$ are reduced. 
    Since $A_0$ and $A_1$ are of finite type over $k$, the ideals $I_0$ and $I_1$ are finitely generated, so there is an integer $N$ such that $f(A_0)\subset \Frob^{-N}(A_1)$, $f(I_0)\subset \Frob^{-N}(I_1)$, $f^{-1}(A_1)\subset \Frob^{-N}(A_0)$ and $f^{-1}(I_1)\subset \Frob^{-N}(I_0)$. 
    Thus, $f$ induces $(A_0)^\wedge_{I_0} \to \Frob^{-N}((A_1)^\wedge_{I_1})$ and $(A_1)^\wedge_{I_1}\to \Frob^{-N}((A_0)^\wedge_{I_0})$, which induce an isomorphism $((A_0)^\wedge_{I_0})^\perf\cong ((A_1)^\wedge_{I_1})^\perf$. 
\end{proof}

The description of distinguished open subspaces of the spectrum of a perfect adic ring tells us that this construction can be glued globally and gives rise to the formal perfect scheme $X^\wedge_Z$.  
In case where $X$ and $Z$ have global deperfections $X_0$ and $Z_0$, $X^\wedge_Z$ is isomorphic to the perfection of $(X_0)^\wedge_{Z_0}$. 
In case $Z$ is a point, we will use the following terminology. 

\begin{defi}
    Let $X$ be a perfect $k$-scheme perfectly locally of finite type and $x\in \lvert X\rvert$ be a closed point. 
    We refer to the ring of global sections of the completion of $X$ at $x$ as the \textit{formal local ring of $X$ at $x$}.  
\end{defi}

\subsection{Dimension theory of local adic rings}\label{ssec:dim}

We will study the dimension of a perfect scheme at each closed point through its formal local ring, or sometimes through its completion. 
The Krull dimension theory can be applied to formal local rings as they are the perfection of Noetherian rings, but it cannot be applied to their completion. 
In this section, we will develop a dimension theory that behaves well also to the completion of formal local rings. 

\begin{defi}
    An adic ring is called local if it is a local ring with the maximal ideal consisting of topologically nilpotent elements. 
\end{defi}

\begin{defi}\label{defi:dim}
Let $A$ be a local adic ring. A sequence of elements $f_1,\ldots,f_n$ in $A$ is said to be a system of parameters if the ideal $(f_1,\ldots,f_n)$ is an ideal of definition. We define the dimension of $A$ as the minimum length of a system of parameters. It is denoted by $\dim A$.
\end{defi}

When $A$ is a local Noetherian ring, it is well-known that this definition is consistent with the Krull dimension. 
What matters here is that this dimension theory using systems of parameters satisfies the invariance under perfection and completion. 

\begin{lem}\label{lem:approxgen}
Let $A$ be a local adic ring and let $I$ be an ideal of definition of $A$ generated by $f_1,\ldots,f_n$. For any elements $a_1,\ldots,a_n$ of $I^2$, $f_1+a_1,\ldots,f_n+a_n$ generate $I$. 
\end{lem}
\begin{proof}
    Let $\mfr{m}_A$ be the maximal ideal of $A$. Then, we have $I=(f_1+a_1,\ldots,f_n+a_n)+I\cdot \mfr{m}_A$. Since $I$ is a finite $A$-module, we may apply Nakayama's lemma to obtain $I=(f_1+a_1,\ldots,f_n+a_n)$. 
\end{proof}
\begin{cor}\label{cor:openofideal}
If the closure of an ideal $I$ of a local adic ring $A$ is open, the ideal $I$ is open. 
\end{cor}
\begin{proof}
    Let $J$ be a finitely generated ideal of definition of $A$ contained in the closure of $I$. For any element $f \in J$, there exists an element $a \in J^2$ such that $f+a\in I$. By \Cref{lem:approxgen}, we have a set of generators of $J$ contained in $I$. Thus, $J\subset I$. 
\end{proof}

\begin{prop}\label{lem:invdimperf}
    The dimension of a local adic ring of characteristic $p$ is invariant under perfection. 
\end{prop}
\begin{proof}
    Let $A$ be a local adic ring of characteristic $p$. Since a system of parameters for $A$ is also that for $A^\perf$, we have $\dim A \geq \dim A^\perf$. We prove the converse inequality. Let ${f}_1,\ldots,{f}_n$ be a set of generators of an ideal of definition $J$ of $A^\perf$, which may be assumed to be in $A$ by taking the $p^N$-th power for sufficiently large $N$. Let $g_1,\ldots,g_m$ be a set of generators of an ideal of definition $I$ of $A$. We may assume that $I\cdot A^\perf\subset J$. Then, we may write $g_i=\sum\limits_{1\leq j\leq n} a_{ij}f_j$ where $a_{ij}\in A^\perf$. Again by taking the $p^N$-th power for sufficiently large $N$, we may assume that $a_{ij}\in A$. Then, $f_1,\ldots,f_n$ generate an ideal of definition of $A$. 
\end{proof}

\begin{prop}\label{lem:invdimcomp}
The dimension of a local adic ring is invariant under completion. 
\end{prop}
\begin{proof}
    Let $A$ be a local adic ring. Since a system of parameters for $A$ is also that for $\cp{A}$, we have $\dim A \geq \dim \cp{A}$. We prove the converse inequality. Let $\hat{f}_1,\ldots,\hat{f}_n$ be a set of generators of an ideal of definition $J$ of $\cp{A}$. By \Cref{lem:approxgen}, we may assume that each $\hat{f}_i$ comes from an element $f_i$ of $A$. Let $g_1,\ldots,g_m$ be a set of generators of an ideal of definition $I$ of $A$. We may assume that $I\cdot \cp{A} \subset J$. Then, there exist elements $a_1,\ldots,a_m$ of $I^2$ such that $g_1+a_1,\ldots,g_m+a_m$ is contained in the ideal $(f_1,\ldots,f_n)$ of $A$. They generate $I$ by \Cref{lem:approxgen}, thus $f_1,\ldots,f_n$ is also a system of parameters for $A$. 
\end{proof}

It is sometimes hard to obtain finiteness of homomorphisms when we deal with non-Noetherian rings. 
As a substitute for it, we will rely on the adicness of homomorphisms to compare dimensions. 

\begin{prop}\label{lem:adicdim}
Let $A$ and $B$ be local adic rings and suppose that we have an adic homomorphism $A\to B$. Then, we have $\dim A \geq \dim B$. 
\end{prop}
\begin{proof}
    Any system of parameters of $A$ is sent to a system of parameters of $B$. 
    Thus, we have $\dim A\geq \dim B$. 
\end{proof}

\begin{lem}\label{lem:adiccrit}
    Let $f\colon A\to B$ be a continuous homomorphism of local adic perfect rings. Let $C$ be the quotient of $B$ by the closure of the ideal $A^{\circ \circ}\cdot B$ in $B$. Then, $f$ is adic if and only if $C$ is a field. 
\end{lem}
\begin{proof}
    If $f$ is adic, $A^{\circ \circ} \cdot B$ contains an ideal of definition of $B$. Since we have $A^{\circ \circ} \cdot B = \Frob(A^{\circ \circ} \cdot B)$ and $f$ is continuous, it follows that $A^{\circ \circ} \cdot B = B^{\circ \circ}$ and $C$ is a field. On the other hand, if $C$ is a field, it follows from \Cref{cor:openofideal} that $A^{\circ \circ}\cdot B$ is an open ideal of $B$. In particular, $A^{\circ \circ} \cdot B$ contains a finitely generated ideal of definition of $B$. Thus, there exists a finitely generated ideal of definition $J$ of $A$ such that $J\cdot B$ is an ideal of definition of $B$. 
\end{proof}

In the above situation, we say that $\Spf(C)$ is the fiber of the closed point of $\Spf(A)$ in $\Spf(B)$. Note that $C$ is a local adic perfect ring with the quotient topology induced from $B$ and is separated with respect to the adic topology. 

\section{Continuity of Witt vectors}\label{sec:2}

The aim of this section is to introduce the notion of continuity of sections of loop groups, positive loop groups and affine Grassmannians on local adic perfect rings.
The crucial result for our application in \Cref{sec:4} is that every section of affine Grassmannians has a lift to loop groups with the same continuity under a mild condition.

\subsection{$F$-standard ideals}\label{ssec:Fstd}

In this section, we introduce a certain class of ideals of perfect rings, which we call $F$-standard. It will be used to measure the continuity of sections of loop groups, positive loop groups and affine Grassmannians. 

\begin{defi}
    An ideal $J$ of a perfect ring $A$ is \textit{$F$-standard} if $J^p=\Frob(J)$. 
\end{defi}

The reason we single out this class of ideals lies in the fact that an $F$-standard ideal $J$ defines an ideal $[J]$ of the ring of Witt vectors that is a mixed characteristic analogue of the ideal $J\llbracket t \rrbracket$ of the ring of formal power series $A\llbracket t \rrbracket$. 

Let $F$ be a local field over $\bb{Q}_p$ with a ring of integers $O$. Let $\pi$ be a uniformizer of $F$ and $k$ be the finite residue field of $O$. 
For a perfect $k$-algebra $A$, let $W_O(A)=W(A)\otimes_{W(k)} O$ and $W_{O,n}(A)=W_O(A)\otimes_{O} O/\pi^n$. 
The Teichm\"{u}ller lift of an element $a\in A$ in $W_O(A)$ is denoted by $[a]$. 

\begin{prop}
    For an $F$-standard ideal $J$ of a perfect $k$-algebra $A$, the subset $[J]$ of $W_O(A)$ consisting of the elements $\sum\limits_{n=0}^\infty [a_n]\pi^n$ with $ a_n\in J$ is an ideal of $W_O(A)$.
\end{prop}
\begin{proof}
    It is clear that $[J]$ is closed under the multiplication by $\pi$ and by $[a]$ for every $a\in A$.
    Thus, it is enough to show that $[J]$ is closed under addition. 
    Consider the equation $\sum\limits_{n=0}^\infty [a_n]\pi^n + \sum\limits_{n=0}^\infty [b_n]\pi^n = \sum\limits_{n=0}^\infty [c_n]\pi^n$ in $W_O(A)$. It is enough to show that $c_n\in J$ whenever $a_n, b_n \in J$ for every $n\geq 0$. By passing to the universal case, we may assume that $A=k[a_0,a_1,\ldots, b_0,b_1,\ldots]^\perf$ and $J\subset A$ is the minimum $F$-standard ideal containing $a_n$ and $b_n$ for every $n\geq 0$. Now, since Teichm\"{u}ller lifts are multiplicative, we have $\sum\limits_{n=0}^\infty [a_0a_n]\pi^n + \sum\limits_{n=0}^\infty [a_0b_n]\pi^n = \sum\limits_{n=0}^\infty [a_0c_n]\pi^n$. It implies that $c_n$ is sent to $a_0c_n$ through the homomorphism $A\to A$ sending $a_n$ (resp.\ $b_n$) to $a_0a_n$ (resp.\ $a_0b_n$). If we set a degree function on $A$ so that the degrees of $a_n^{1/p^k}$ and $b_n^{1/p^k}$ are $p^{-k}$ for every $n,k \geq 0$, we see that $c_n$ is homogeneous of degree $1$. Thus, for some sufficiently large $N\geq 0$, we have $c_n^{p^N} \in (a_0,a_1,\ldots,b_0,b_1,\ldots)^{p^N}\subset J^{p^N} = \Frob^{N}(J)$ and we see that $c_n\in J$. 
\end{proof}

The ideal $[J]$ enables us to define congruence modulo $J$ of $W_O(A)$-valued points of a suitable presheaf $X$. 
For an element $x\in X(W_O(A))$, its restriction to $W_O(A)/[J]$ is denoted by $x\vert_{W_O(A)/[J]}$. We say that two elements $x,x'\in X(W_O(A))$ are congruent modulo $J$ if $x\vert_{W_O(A)/[J]}=x'\vert_{W_O(A)/[J]}$.
To see that this congruence works well, we show that every adic perfect ring has an $F$-standard ideal of definition. 

\begin{defi}
    For an ideal $I$ of a perfect ring $A$, the minimum $F$-standard ideal containing $I$, which is $\bigcup\limits_{n\geq 0} \left(\Frob^{-n}(I)\right)^{p^n}$, is called the \textit{$F$-standardization} of $I$. 
\end{defi}

\begin{lem} \label{lem:Fstdfunc}
    Let $f\colon A\to B$ be a homomorphism of perfect rings. Let $I$ be an ideal of $A$ and let $J=I\cdot B$. Let $\wtd{I}$ (resp.\ $\wtd{J}$) be the $F$-standardization of $I$ (resp.\ $J$). We have $\wtd{J}=\wtd{I}\cdot B$. In particular, $J$ is $F$-standard if $I$ is $F$-standard. 
\end{lem}
\begin{proof}
    We have $\wtd{J}=\bigcup\limits_{n\geq 0} \left(\Frob^{-n}(I\cdot B)\right)^{p^n} = \bigcup\limits_{n\geq 0} \left(\Frob^{-n}(I)\right)^{p^n} \cdot B = \wtd{I}\cdot B$. 
\end{proof}

\begin{lem}\label{lem:Fstdpre}
    For an ideal of definition $I$ of an adic perfect ring $A$, the $F$-standardization of $I$ is an ideal of definition of $A$.
\end{lem}
\begin{proof}
    We may replace $I$ with a finitely generated ideal of definition. 
    Let $f_1,f_2,\ldots,f_m$ be a finite set of generators of $I$. The $F$-standardization $J$ of $I$ is the union of the ideals $\bigl(f_1^{1/p^n},f_2^{1/p^n},\ldots,f_m^{1/p^n}\bigr)^{p^n}$. We have $\bigl(f_1^{1/p^n},f_2^{1/p^n},\ldots,f_m^{1/p^n}\bigr)^{mp^n} \subset \left(f_1,f_2,\ldots,f_m\right)$, so $J^m\subset I \subset J$ and $J$ is an ideal of definition of $A$. 
\end{proof}

\begin{cor}
The completion of an adic perfect ring is perfect. 
\end{cor}
\begin{proof}
Let $A$ be a perfect adic ring and let $J$ be an $F$-standard ideal of definition. The Frobenius map induces an isomorphism $A/J^n \cong A/J^{pn}$. Thus, the Frobenius map on $A^\wedge$ is an isomorphism. 
\end{proof}

The same argument works for a bit weaker notion, which we call topologically nilpotent. 

\begin{defi}
    We say that an ideal of an adic ring $A$ is \textit{topologically nilpotent} if it is contained in an ideal of definition of $A$. 
\end{defi}

\begin{prop}
    For an ideal $I$ of an adic ring $A$, the following are equivalent. 
    \begin{enumerate}
        \item The ideal $I$ is topologically nilpotent.
        \item For some positive integer $n$, $I^n$ is topologically nilpotent. 
        \item For every ideal of definition $J$ of $A$, there is a positive integer $n$ such that $I^n\subset J$. 
    \end{enumerate}
\end{prop}

\begin{proof}
It is easy to see that $(1)$ implies $(2)$ and $(2)$ implies $(3)$. Suppose the condition $(3)$. 
Then, for every ideal of definition $J$, the ideal $I+J$ is an ideal of definition of $A$, so
$(3)$ implies $(1)$. 
\end{proof}

\begin{lem}
    For a topologically nilpotent ideal $I$ of an adic perfect ring $A$, the $F$-standardization of $I$ is topologically nilpotent. 
\end{lem}
\begin{proof}
It follows from \Cref{lem:Fstdpre} because $I$ is contained in an ideal of definition of $A$. 
\end{proof}

\subsection{Loop groups, positive loop groups and affine Grassmannians}\label{ssec:contsec}
In this section, we define and study the continuity of sections of loop groups, positive loop groups and affine Grassmannians on local adic perfect rings. 

Let $G$ be a reductive group over $F$ and $\cl{G}$ be a smooth affine model of $G$ over $O$. We do not need to assume that $\cl{G}$ is reductive. 
The loop group $LG$ and the positive loop group $L^+\cl{G}$ are fpqc sheaves on perfect $k$-algebras that send a perfect $k$-algebra $A$ to the group $G(W_O(A)\invp)$ and $\cl{G}(W_O(A))$, respectively. 
The affine Grassmannian $\GR_{\cl{G}}$ is the quotient $LG/L^+\cl{G}$ as a fpqc sheaf. 
For a section $g$ of $LG$, its class in $\GR_\cl{G}$ is denoted by $[g]$. 

Let $\kappa$ be a perfect field over $k$ and let $A$ be a local adic perfect $\kappa$-algebra with a residue field isomorphic to $\kappa$. 
For a sheaf $X\in \{LG,L^+\cl{G},\GR_\cl{G}\}$ and a section $x\in X(A)$, we measure the difference between $x$ and its pullback $x_\kappa \in X(\kappa)$ via the congruence modulo $F$-standard ideals. The pullback $x_\kappa$ is called the reduction of $x$. For an $O$-algebra $R$, the pullback of an element $g\in \cl{G}(R)$ to an $R$-algebra $S$ is denoted by $g\vert_S$. 

\begin{defi}
An $A$-valued point $g\in G(W_O(A)\invp)$ of $LG$ is \textit{continuous}
if there exists a topologically nilpotent $F$-standard ideal $I$ of $A$ such that $g\vert_{W_O(A)\invpp/[I]}$ is equal to $g\vert_{W_O(\kappa)\invpp}$ in $G(W_O(A)[\tfrac{1}{p}]/[I])$.
An $A$-valued point $g\in \cl{G}(W_O(A))$ of $L^+\cl{G}$ is \textit{continuous}
if there exists a topologically nilpotent $F$-standard ideal $I$ of $A$ such that $g\vert_{W_O(A)/[I]}$ is equal to $g\vert_{W_O(\kappa)}$ in $\cl{G}(W_O(A)[\tfrac{1}{p}]/[I])$.
\end{defi}

Here, we identify $g\vert_{W_O(\kappa)[1/p]}$ (resp.\ $g\vert_{W_O(\kappa)}$) with an element of $G(W_O(A)[\tfrac{1}{p}]/[I])$ (resp.\ $\cl{G}(W_O(A)/[I])$) via the pullback along $W_O(\kappa)[1/p] \to W_O(A)[\tfrac{1}{p}]/[I]$ (resp.\ $W_O(\kappa) \to W_O(A)/[I]$). 
When $g$ satisfies the given condition for a topologically nilpotent $F$-standard ideal $I$, we say that $g$ is $I$-continuous. 
As we see in the following lemma, the set of such ideals $I$ has a unique minimal element if it is non-empty. 

\begin{lem}\label{lem:contloop}
    For a continuous $A$-valued point $g$ of $LG$ or of $L^+\cl{G}$, the set of topologically nilpotent $F$-standard ideals $I$ such that $g$ is $I$-continuous has a unique minimal element. The minimal element is called the continuity of $g$. 
\end{lem}
\begin{proof}
    The $I$-continuity of $g$ is equivalent to that $x=g^{-1}g_\kappa$ is congruent to the identity modulo $I$. 
    Let $f_1,\ldots,f_n$ be a set of generators of $\Gamma(\cl{G},\cl{O}_{\cl{G}})$ as an $O$-algebra such that $f_i(\mrm{id}_{\cl{G}})=0$ for all $i$. 
    The triviality of $x$ modulo $I$ is equivalent to the triviality of all coefficients of $f_i(x)$ modulo $I$ for all $i$. 
    By the continuity of $g$, these coefficients lie in a topologically nilpotent $F$-standard ideal. 
    Thus, the $F$-standard ideal generated by the coefficients of $f_i(x)$ is the desired minimal element. 
\end{proof}

It turns out that every $A$-valued point of $\GR_\cl{G}$ (coming from $LG$) is continuous in a similar sense and we can define its continuity. This reflects the fact that affine Grassmannians are of ind-finite type. 

\begin{prop}\label{prop:radideal}
    For every $g\in G(W_O(A)\invp)$, there exists a topologically nilpotent $F$-standard ideal $I$ of $A$ such that $g\vert_{W_O(A)\invpp/[I]}$ and $g\vert_{W_O(\kappa)\invpp}$ represent the same class in $G(W_O(A)\invp/[I])/\cl{G}(W_O(A)/[I])$. Moreover, the set of such ideals $I$ has a unique minimal element. The minimal element is called the continuity of $[g]$. 
\end{prop}
\begin{proof}
    Take $x$ and $f_1,\ldots,f_n$ as in the proof of \Cref{lem:contloop}. 
    The congruence of $[g]$ and $[g_\kappa]$ modulo $I$ is equivalent to $x\vert_{W_O(A)\invpp/[I]}\in \cl{G}(W_O(A)/[I])$. 
    This condition is equivalent to the triviality of the negative coefficients of $f_i(x)$ modulo $I$ for all $i$. 
    There are only finitely many nontrivial negative coefficients and they are all topologically nilpotent. Thus, the $F$-standard ideal generated by the negative coefficients of $f_i(x)$ is the desired minimal element. 
\end{proof}

We see from the construction that for a sheaf $X\in \{LG,L^+\cl{G},\GR_\cl{G}\}$ and a section $x\in X(A)$, the continuity of $x$ is functorial in $A$. 

\begin{lem}\label{rem:bcofideal}
    Let $A$ and $B$ be local adic perfect $\kappa$-algebras with a residue field isomorphic to $\kappa$. Let $f\colon A\to B$ be a homomorphism of $\kappa$-algebras. 
    For every continuous $A$-valued point $g$ of $LG$, of $L^+\cl{G}$, or of $\GR_\cl{G}$ coming from that of $LG$, $f^*g$ is continuous and the continuity of $f^*g$ is generated by the image of the continuity of $g$ along $f$. 
\end{lem}
\begin{proof}
    Since the residue fields of $A$ and $B$ are both isomorphic to $\kappa$, we see that $f$ is local. Thus, the image of a finitely generated ideal of definition of $A$ along $f$ generates a topologically nilpotent ideal of $B$, so $f$ is continuous. 
    
    Let $I$ be the continuity of $f$ and let $J=I\cdot B$. Since $I$ is topologically nilpotent and $f$ is continuous, $J$ is topologically nilpotent. Moreover, $J$ is $F$-standard by \Cref{lem:Fstdfunc}. Since we have a homomorphism $W_O(A)/[I] \to W_O(B)/[J]$, we see that $f^*g$ is $J$-continuous. Now, take $x$ and $f_1,\ldots,f_n$ as in the proof of \Cref{lem:contloop}. The continuity of $g$ and $f^*g$ is generated by certain coefficients of $f_i(x)$ and $f_i(f^*x)$, respectively. Since we have $f_i(f^*x)=W_O(f)(f_i(x))$, the continuity of $f^*g$ is equal to $J$. 
\end{proof}

For a continuous element $g\in G(W_O(A)\invp)$, the continuity of $[g]$ is contained in the continuity of $g$, but the converse inclusion does not hold in general. 
However, under a mild assumption, we can show that every $A$-valued point of $\GR_\cl{G}$ has a representative in $LG$ with the same continuity. 

\begin{prop}\label{prop:modifpoint}
Suppose that $A$ is the perfection of a complete local Noetherian $\kappa$-algebra with a residue field isomorphic to $\kappa$.
For every $g\in G(W_O(A)\invp)$, there exists a representative of $[g]$ in $G(W_O(A)\invp)$ with the same continuity as $[g]$. 
\end{prop}
\begin{proof}
    Let $I$ be the continuity of $[g]$. Let $x=g^{-1}g_\kappa$ and $a=x\vert_{W_O(A)\invpp/[I]}$. We have $a\in \cl{G}(W_O(A)/[I])$ by the definition of $I$. If we can lift it to an element of $\cl{G}(W_O(A))$, the right translation of $g$ by the lift is a desired representative. Thus, it is enough to construct a lift of $a$. We successively construct its $\pi^n$-approximation $a_n\in \cl{G}(W_{O,n}(A))$. First, we consider the case $n=1$. In this case, it is enough to show that the map $\cl{G}(A)\to \cl{G}(A/I)$ is surjective. Since $\cl{G}$ is finitely presented, an $(A/I)$-valued point of $\cl{G}$ can be lifted to an $(A_0/(I\cap A_0))$-valued point for some deperfection $A_0$ of $A$. Since $A_0$ is a complete local Noetherian $\kappa$-algebra, and thus $(I\cap A_0)$-adically complete, such a lift can also be lifted to an $A_0$-valued point by the smoothness of $\cl{G}$. Then, we can pull back such a lift to an $A$-valued point. Next, suppose that we have a lift $a_n$ for some $n \geq 1$. Take a lift $a'$ of $a_n$ to a $W_{O,n+1}(A)$-valued point. The difference between $a$ and $a'$ in $\cl{G}(W_{O,n+1}(A)/[I])$ is trivial modulo $\pi^n$, so it corresponds to an element of $\Lie_{A/I}\hsp{3pt}\cl{G}$. Since $\Lie_{A}\hsp{3pt}\cl{G} \to \Lie_{A/I}\hsp{3pt}\cl{G}$ is surjective, the difference can be lifted to a $W_{O,n+1}(A)$-valued point trivial modulo $\pi^n$. By modifying $a'$ by the lift, we have a lift $a_{n+1}$ compatible with $a_n$. Finally, we can take $\lim\limits_{n\to \infty} a_n$ as a lift of $a$. 
\end{proof}

This representative in $LG$ can be characterized functorially through \Cref{rem:bcofideal}. This functorial characterization is crucial for our application in \Cref{sec:4}. 

\begin{cor}\label{thm:constdef}
Suppose that $\kappa$ is algebraically closed and $A$ is the perfection of a complete local Noetherian $\kappa$-algebra with a residue field isomorphic to $\kappa$. 
For every $A$-valued point $[g]\in \GR_\cl{G}(A)$, there exists a continuous representative $g\in G(W_O(A)\invp)$ such that 
for every local adic perfect $A$-algebra $B$ with a residue field isomorphic to $\kappa$ and every $F$-standard ideal $J$ of $B$, the congruence of $[g]_B$ and $[g]_\kappa$ modulo $J$ is equivalent to the congruence of $g_B$ and $g_\kappa$ modulo $J$. 
\end{cor}
\begin{proof}
We may write $A=R^\perf$ with $R$ a complete local Noetherian $\kappa$-algebra
with a residue field isomorphic to $\kappa$. 
Since $\kappa$ is algebraically closed, $R$ is strictly henselian, so 
its perfection $A$ is also strictly henselian. 
Thus, any $\cl{G}$-torsor on $W_O(A)$ is trivial, and 
we may take a representative $g$ of $[g]$ in $G(W_O(A)\invp)$. 

Let $I$ be the continuity of $[g]$.  
By \Cref{prop:modifpoint}, we may take a representative $g$ with the continuity $I$. 
We show that this choice of $g$ satisfies the given condition. 
By \Cref{rem:bcofideal}, the continuity of $[g]_B$ and that of $g_B$ are both equal to $I\cdot B$, so we have the desired equivalence. 
\end{proof}

\section{Fundamental alcoves}\label{sec:3}
In this section, we review the notion of fundamental alcoves developed in \cite{GHKR10} and \cite[Section 6]{Vie14}. It is a key technical ingredient of the argument in \cite{HV12}. Note that \cite[Section 6]{Vie14} only deals with the case $F=W(\bb{F}_q)[\tfrac{1}{p}]$ for some $q=p^r$ in mixed characteristic, but its argument works for any non-archimedean local field $F$. 

We keep the notation in \Cref{sec:2} and assume from now on that $\cl{G}$ is reductive. 
Moreover, we fix an algebraic closure $\ov{k}$ of $k$ and work over the category of perfect $\ov{k}$-algebras when we consider loop groups, positive loop groups and affine Grassmannians. 
Let $q$ be the cardinality of $k$. For a perfect $k$-algebra $A$, the relative Frobenius on $W_O(A)$ (taking $q$-th power of each coefficient) is denoted by $\sigma$. Its action on $LG$ and $L^+\cl{G}$ is also denoted by $\sigma$. 
Let $L=W_O(\ov{k})\invp$ and $O_L=W_O(\ov{k})$. 

Fix a Borel pair $T \subset B \subset G$ and consider the apartment $\cl{A}(G_L,T_L)$ of the Bruhat-Tits building of $G_L$. 
We set the origin of the apartment $\cl{A}(G_L,T_L)$ to the hyperspecial point corresponding to $\cl{G}$ and set the positive chamber to the one corresponding to $B_L$. The negative chamber contains a unique alcove containing the origin. We fix an Iwahori subgroup to the one corresponding to the alcove and denote its positive loop group by $I$. 

A parabolic subgroup of $G_L$ containing $T_L$ (not necessarily defined over $F$) is called semistandard and usually denoted by $P$. Its unipotent radical is denoted by $N$ and its Levi subgroup containing $T_L$ is denoted by $M$. The unipotent radical of the opposite of $P$ is denoted by $\ov{N}$. The Iwahori subgroup $I$ has an Iwahori decomposition $I=I_NI_MI_{\ov{N}}$. Here, $I_H$ denotes the intersection $I\cap LH$ for a subgroup $H$ of $G_L$. 

Let $\cl{T}$ be the identity component of the N\'{e}ron model of $T$ and let $N_T$ be the normalizer of $T$ in $G$. The extended affine Weyl group $N_T(L)/\cl{T}(O_L)$ is denoted by $\wtd{W}$ and that of a Levi subgroup $M$ is denoted by $\wtd{W}_M$. The length function on $\wtd{W}$ with respect to $I$ is denoted by $\ell$. For $x \in G(L)$, let $\phi_x\colon LG\to LG$ be the map sending $g$ to $\sigma(xgx^{-1})$. 

\begin{defi}\textup{(\cite[Definition 6.1]{Vie14})}
An element $x$ of $\wtd{W}$ is \textit{$P$-fundamental} for a semistandard parabolic subgroup $P$ if $\phi_x(I_M)=I_M$, $\phi_x(I_N)\subset I_N$ and $\phi_x^{-1}(I_{\ov{N}})\subset I_{\ov{N}}$. 
\end{defi}

We fix a positive integer $r$ so that $T$ splits over the unramified extension of $F$ of degree $r$. The action of $\sigma^r$ on $\wtd{W}$ is trivial. Let $x\in \wtd{W}$ be a $P$-fundamental element and let $x'=\sigma(x)\sigma^2(x)\cdots\sigma^r(x)$. Since $\phi_x^r$ preserves $LP$, we have $x'Px'^{-1}=P$ and $x'\in \wtd{W}_M$. Moreover, since $\phi_x^r$ preserves the Iwahori subgroup $I_M$, some power of $x'$ is central in $M$. Thus, we may take a representative of $x$ in $N_T(L)$ so that it is decent in the sense of \cite[Definition 1.8]{RZ96}. The $M$-dominant Newton point of the $\sigma^r$-conjugacy class of $x'$ is denoted by $\nu_{r,x}$. It is central in $M$ because some power of $x'$ is central in $M$. 

\begin{lem}\textup{(\cite[Lemma 6.3]{Vie14})}
Let $M'$ be the centralizer of $\nu_{r,x}$. The semistandard parabolic subgroup $P'=M'N$ is the largest semistandard parabolic subgroup for which $x$ is fundamental. 
\end{lem}

As we see from the following propositions, fundamental alcoves play an important role as well-behaved representatives of $\sigma$-conjugacy classes of $G$. 

\begin{lem}\textup{(\cite[Lemma 6.4]{Vie14})}\label{lem:sigmaconjI}
Every $\ov{k}$-valued point of $IxI$ is $\sigma$-conjugate to $x$ by a $\ov{k}$-valued point of $I$. Every $\ov{k}$-valued point of $xI_M$ is $\sigma$-conjugate to $x$ by a $\ov{k}$-valued point of $xI_Mx^{-1}$. 
\end{lem}

\begin{prop}\textup{(\cite[Theorem 6.5]{Vie14})}
Every $\sigma$-conjugacy class in $G$ has a representative in $\wtd{W}$ that is $P$-fundamental for some $P$. 
\end{prop}

From now on, when we say that $x$ is $P$-fundamental, we assume that $x$ is decent and $M$ is the centralizer of $\nu_{r,x}$. Then, $\{\phi_x^i(I_N)\}_{i\geq 0}$ (resp.\ $\{\phi_x^{-i}(I_{\ov{N}})\}_{i\geq 0}$) is convergent to $1$ as a sequence of subgroups of $I_N$ (resp.\ $I_{\ov{N}}$) with respect to the $\pi$-adic topology. Under this assumption, we have the following variant of \Cref{lem:sigmaconjI}. 

\begin{lem}\textup{(\cite[Lemma 6.4]{Vie14})}\label{lem:sigmaconjII}
Every $\ov{k}$-valued point of $x\phi_x^i(I_N)$ \textup{(}resp.\ $x\phi_x^{-i}(I_{\ov{N}})$\textup{)} is $\sigma$-conjugate to $x$ by a $\ov{k}$-valued point of $x\phi_x^i(I_N)x^{-1}$ \textup{(}resp.\ $\sigma^{-1}(\phi_x^{-i}(I_{\ov{N}}))$\textup{)}. 
\end{lem}

In \Cref{sec:4}, we need the following lemma to carry out the argument in \cite{HV12}. Since it is only proved when $G$ is split in the literature, we record its proof that works in general when $G$ is unramified.  

\begin{lem}\textup{(\cite[Lemma 4.5]{HV12})}\label{lem:IxIisclsd}
As a subfunctor of $LG$, $IxI$ is closed in $[x]$. Here, the subfunctor $[x]$ consists of those sections of $LG$ that are $\sigma$-conjugate to $x$ at each geometric point. 
\end{lem}
\begin{proof}
The closure of $IxI$ in $LG$ is the union of $IyI$ with $y\leq x$ in the Bruhat order. Suppose that $IyI$ meets $[x]$ with $y\leq x$. The Bruhat-Tits decomposition gives $w\in \wtd{W}$ and $i\in I$ satisfying $(iw)^{-1}x\sigma(iw) \in IyI$.  Let $n$ be an arbitrary positive integer and let $x'=x\sigma(x)\cdots \sigma^{n-1}(x)$. The element $(iw)^{-1}x'\sigma^{n}(iw)$ lies in $IyI\sigma(y)I\cdots I\sigma^{n-1}(y)I$, and we have $n\ell(y)+2\ell(w)\geq \ell(x')$. It is enough to show that $\ell(x')=n\ell(x)$ because it implies that $\ell(y)=\ell(x)$ and $y=x$.

We prove that $IxI\sigma(x)I\cdots I\sigma^{n-1}(x)I=Ix'I$ by induction on $n$. Suppose that it holds for some $n$. Since $\phi_x^n(I_P)\subset I_P$ and $\phi_x^{-1}(I_{\ov{N}})\subset I_{\ov{N}}$, we have $x'I_{\sigma^{n-1}(P)}x'^{-1}\subset I_{\sigma^{-1}(P)}$ and $\sigma^{n}(x)^{-1}I_{\sigma^{n-1}(\ov{N})}\sigma^n(x)\subset I_{\sigma^{n}(\ov{N})}$. It implies that $Ix'I\sigma^n(x)I=Ix'\sigma^n(x)I$ by the Iwahori decomposition $I=I_{\sigma^{n-1}(P)}I_{\sigma^{n-1}(\ov{N})}$, and we have the claim for $n+1$. 
\end{proof}
\section{Proof of equidimensionality}\label{sec:4}

The aim of this section is to construct a local foliation as a correspondence and deduce the equidimensionality of $X_{\leq \mu}(b)$. 

\subsection{Affine Deligne-Lusztig varieties}\label{ssec:ADLV}

In this section, we recall the definition of affine Deligne-Lusztig varieties. 
We keep the notation in \Cref{sec:3}. 

Affine Deligne-Lusztig varieties are certain perfect schemes associated with the triple $(\cl{G},b,\mu)$ that are locally perfectly of finite type over $\ov{k}$. Here, $b\in G(L)$ is an element specifying a $\sigma$-conjugacy class $[b]\in B(G)$ and $\mu\in X_*(T)_+$ is a dominant cocharacter. 
Regarded as a sheaf over the category of perfect $\ov{k}$-schemes, the closed affine Deligne-Lusztig variety $X_{\leq \mu}(b)$ is a closed subsheaf of $\GR_\cl{G}$ described as
\[X_{\leq \mu}(b)=\Bigl\{g\in \GR_\cl{G} \Bigm\vert g^{-1}b\sigma(g) \in [L^+\cl{G}\backslash LG/L^+\cl{G}]_{\leq \mu} \Bigr\}.\]
Here, the stack $[L^+\cl{G}\backslash LG/L^+\cl{G}]_{\leq \mu}$ is a substack of $[L^+\cl{G}\backslash LG/L^+\cl{G}]$ classifying modifications of $\cl{G}$-torsors bounded by $\mu$. 
The affine Deligne-Lusztig variety $X_\mu(b)$ is the open subscheme of $X_{\leq \mu}(b)$ where $g^{-1}b\sigma(g)$ lies in $[L^+\cl{G}\backslash LG/L^+\cl{G}]_\mu$, the locus where the differences of modifications are exactly equal to $\mu$. 
It is known that $X_\mu(b)$ is non-empty if and only if $[b]\in B(G,\mu)$ (see \cite[Theorem 5.1]{Gas10}). 

A dimension formula of $X_{\leq \mu}(b)$ was obtained by \cite{GHKR06}, \cite{Vie06} and \cite{Ham15} in equal characteristic and the same arguments were verified to work in mixed characteristic in \cite{Zhu17}. It goes as follows: 
\[\dim X_{\leq \mu}(b)=\langle\rho,\mu-\nu\rangle-\deff(b)/2. \]
Here, $\nu$ is the dominant Newton point of $b$ and $\deff(b)$ is the defect of $b$ (see \cite{GHKR06}). 
We will deduce the equidimensionality of $X_{\leq \mu}(b)$ by giving a lower bound of the dimension at each closed point, which turns out to be equal to the right-hand side of the formula. 

\subsection{Strategy of a proof}
In this section, we review the strategy of a proof in equal characteristic in \cite{HV12} and explain how to translate it into mixed characteristic. 

The main idea is to develop local foliations, or almost product structures, to reduce to the equidimensionality of Newton strata. 
First, let us recall Oort's foliations in \cite{Oor04} briefly. It roughly states that a Newton stratum $\ov{\mrm{Sh}}^b$ of the special fiber of a Shimura variety can be decomposed into the underlying space $\mrm{RZ}^{b,\red}$ of a Rapoport-Zink space and an Igusa variety $\mrm{Igs}^b$. More precisely, there is a finite surjective map $\mrm{RZ}^{b,\red} \times \mrm{Igs}^b\to \ov{\mrm{Sh}}^b$ and this is what we call Oort's foliation. Since $\mrm{Igs}^b$ is smooth and equidimensional, the equidimensionality of $\mrm{RZ}^{b,\red}$ is reduced to that of $\ov{\mrm{Sh}}^b$. 

The problem is that Shimura varieties are not accessible any more when $\mu$ is non-minuscule. The motivation for introducing local foliations is to resolve this issue by taking the completion at a closed point. Via the Serre-Tate theory, it makes the target space $\ov{\mrm{Sh}}^b$ into the closed Newton stratum of the universal deformation space of a certain local geometric object involved with $p$-divisible groups or local shtukas. In particular, local foliations can be purely expressed in terms of these local geometric objects. In \cite{HV12}, they realized this picture with local shtukas. 

In mixed characteristic, we will work with $F$-crystals with $\cl{G}$-structure since they are the local geometric objects parametrized by $X_{\leq \mu}(b)$. This is why we need a theory of formal perfect algebraic geometry developed in \Cref{sec:1} and \Cref{sec:2}. 
However, we still have a serious problem about what would be the universal deformation in this context. We avoid this problem by replacing it with a certain deformation imitating the construction of the universal deformation in \cite{HV11}. 
Due to this lack of universality, our local foliations are not given by finite surjective maps. In \Cref{ssec:compdim}, we will construct local foliations as correspondences between relevant spaces. They turn out to be enough for our purpose to compare the dimensions of both sides. 

Take a closed point $[g]$ of $X_{\leq \mu}(b)$ and fix a representative $g\in G(L)$. We will study the completion $X_{\leq \mu}(b)^\wedge_{[g]}$ of $X_{\leq \mu}(b)$ at $[g]$ to calculate the dimension at $[g]$. 

\subsection{A deformation and the Newton stratification}\label{ssec:Newton}

In this section, we construct a deformation of the $F$-crystal with $\cl{G}$-structure associated with $[g]$ that serves as a replacement for its universal deformation. 

The $F$-crystal with $\cl{G}$-structure associated with $[g]$ is written as $(\cl{G},b'\sigma)$ with $b'=g^{-1}b\sigma(g)$. 
We imitate the construction of its deformation in \cite[Theorem 5.6]{HV11}. 
Let $w_0$ be the longest element of the Weyl group and set $\mu^*=w_0(-\mu)$. 
Let $A$ be the formal local ring of $\GR_{\cl{G},\leq \mu^*}$ at $[b'^{-1}]$. It is of dimension $\langle 2\rho,\mu\rangle$ as in \cite[Proposition 1.23]{Zhu17}. 
There is a natural map $\Spec(A) \to \GR_{\cl{G},\leq \mu^*}$ and we apply \Cref{thm:constdef} to the $A$-valued point of $\GR_\cl{G}$ to obtain a representative $t\in G(W_O(A)\invp)$. We may assume that its reduction is equal to $b'^{-1}$. Then, we obtain a deformation $(\cl{G},t^{-1}\sigma)$ of $(\cl{G},b'\sigma)$ over $A$. 

We will introduce the closed Newton stratum defined by the deformation. 
As in \cite[Theorem 3.6]{RR96}, we take the Newton stratification of $\Spec(A)$ corresponding to the $F$-crystal with $G$-structure $(G,t^{-1}\sigma)$. 
Let $N^b$ be the closed Newton stratum, the unique one containing the closed point corresponding to the Newton point $\nu$. Note that the Kottwitz map is locally constant as in \cite[Proposition 1.21]{Zhu17}, so the $\sigma$-conjugacy class attached to each geometric point of $N^b$ is equal to $[b]$. 
The same argument as in \cite[Proposition 7.8]{HV11} works in this setting and gives the following lower bound of the dimension of $N^b$. 

\begin{prop}\label{prop:dimlwb}
$\dim N^b \geq \langle\rho,\mu+\nu\rangle-\deff(b)/2$. 
\end{prop}

Here, we rely on the purity of the Newton stratification over a perfect (thus non-Noetherian) ring $A$ that can be deduced from \cite[Main Theorem B]{Vas06} as in \cite[Theorem 7.4]{HV11}. We also use the fact that $A$ is the perfection of a complete local Noetherian ring so that its spectrum behaves like the spectrum of a complete local Noetherian ring.

\subsection{A local foliation as a correspondence}\label{ssec:compdim}

In this section, we will construct a local foliation relating the completion $X_{\leq \mu}(b)^\wedge_{[g]}$ of $X_{\leq \mu}(b)$ at $[g]$ to the closed Newton stratum $N^b$. Our local foliation is constructed as the following correspondence. 

\begin{center}
    \[
    \begin{tikzcd}
        &(\wtd{N}^b)^{\wedge}\arrow[ld]\arrow[rd]& \\
        X_{\leq \mu}(b)^\wedge_{[g]}\mathrel{\widehat{\times}} I_{x,n}^\wedge && N^b
    \end{tikzcd}
    \]
\end{center}
We will later introduce the relevant spaces $\wtd{N}^b$ and $I_{x,n}$, but let us stress here that this correspondence does not consist of finite surjective maps. 
We will see that $\wtd{N}^b$ is integral and perfectly of finite type over $N^b$, but it does not directly map to $X_{\leq \mu}(b)^\wedge_{[g]}\widehat{\times} I_{x,n}^\wedge$ due to its incompleteness. We cannot expect the map from $(\wtd{N}^b)^{\wedge}$ to the product to be integral, but instead we can show that it is adic. This is where the dimension theory developed in \Cref{ssec:dim} comes into our argument. 

We will follow the construction in the proof of \cite[Theorem 6.5]{HV12}. 
Let us write $N^b=\Spec(A^b)$. Here, $A^b$ is the perfection of a complete local Noetherian ring by construction. 
Fix a $P$-fundamental element $x\in \wtd{W}$ contained in $[b]$. 

\begin{prop}\label{prop:firstchoice}
    There is a perfect local adic $A^b$-algebra $\wtd{A}^b$ and a continuous section $h\in LG(\wtd{A}^b)$ such that 
    \begin{itemize}
        \item the homomorphism $A^b\to \wtd{A}^b$ is the perfection of a finite homomorphism between complete local Noetherian rings of the same dimension, and
        \item $h^{-1}t^{-1}\sigma(h)$ lies in $IxI$. 
    \end{itemize}
\end{prop}
\begin{proof}
    This is a mixed characteristic analogue of \cite[Theorem 4.14]{HV12} and the same proof can be applied in our context as well. It is worked out when $G$ is split, but it does not matter since we have \Cref{lem:IxIisclsd} in the general unramified case. 
\end{proof}

In the previous diagram, $\wtd{N}^b$ is the spectrum of $\wtd{A}^b$, so it is in a sense finite and surjective over $N^b$. The section $h$ will play the role of a quasi-isogeny between $(\cl{G},t^{-1}\sigma)$ and those classified by the space $I^\wedge_{x,n}$. However, to further this argument, we need to take the completion of $\wtd{A}^b$. 
Let $(\wtd{A}^b)^\wedge$ be the completion of the perfect adic ring $\wtd{A}^b$.

\begin{prop}\label{prop:secondchoice}
    There is a continuous section $h\in LG((\wtd{A}^b)^\wedge)$ such that $h^{-1}t^{-1}\sigma(h)$ lies in $xI_{\ov{N}}$. 
\end{prop}
\begin{proof}
    This is the content of \cite[Corollary 4.15]{HV12} in equal characteristic. Here, we record a mixed characteristic translation of the argument to show the necessity of taking the completion. Note that the treatment of a Levi factor is a bit simplified from \cite[Proposition 8.1]{HV11}. 

    First, we have $IxI/I\cong I/(I\cap xIx^{-1}) \cong I/(I\cap \phi_x(I)) \cong I_N/\phi_x(I_N)$. 
    Here, the second isomorphism is given by the Frobenius map on $I$. 
    Take $h\in LG(\wtd{A}^b)$ as in \Cref{prop:firstchoice}. 
    Since the quotient $I_N\to I_N/\phi_x(I_N)$ admits a scheme-theoretic section, 
    there exists a continuous section $n\in I_N(\wtd{A}^b)$ such that
    $[\sigma^{-1}(n)x]=[h^{-1}t^{-1}\sigma(h)]$ in $IxI/I$. 
    By replacing $h$ with $h\sigma^{-1}(n)$, we may assume that $h^{-1}t^{-1}\sigma(h)\in xI$. 

    Let $J$ be an ideal of definition of $\wtd{A}^b$ such that $h$ and $t$ are $J$-continuous. 
    We construct a sequence $\{h_i\}_{i\geq 0}$ of $J$-continuous $\wtd{A}^b$-valued points of $LG$ such that $h_i^{-1}t^{-1}\sigma(h_i)\in x\phi_x^i(I_N)I_MI_{\ov{N}}$. We set $h_0=h$. Suppose that $h_i^{-1}t^{-1}\sigma(h_i)=xnm\bar{n}$ with $n\in \phi_x^i(I_N)$, $m\in I_M$ and $\bar{n}\in I_{\ov{N}}$. If we set $h_{i+1}=h_ixnx^{-1}$, we have $h_{i+1}^{-1}t^{-1}\sigma(h_{i+1})=xm\bar{n}\phi_x(n)$. Since $m\bar{n}\phi_x(n) \in \phi_x^{i+1}(I)\cap I$, $h_{i+1}^{-1}t^{-1}\sigma(h_{i+1})$ lies in $x\phi_x^{i+1}(I_N)I_MI_{\ov{N}}$. If $h_i$ is $J$-continuous, then $n$ is $J$-continuous, so $h_{i+1}$ is also $J$-continuous. 

    By construction, $h_i^{-1}h_{i+1}$ lies in $x\phi_x^i(I_N)x^{-1}$. Since $\{\phi_x^i(I_N)\}_{i\geq 0}$ is convergent to $1$, we can take the limit $h'$ of the sequence $\{h_i\}_{i\geq 0}$ that is also $J$-continuous. By replacing $h$ with $h'$, we may assume that $h^{-1}t^{-1}\sigma(h)\in xI_MI_{\ov{N}}$. 

    Let us write $h^{-1}t^{-1}\sigma(h)=xm\bar{n}$. We show that $xm$ is $\sigma$-conjugate to $x$ by the translation of a continuous $(\wtd{A}^b)^\wedge$-valued point of $I_M$. By translating by a $\ov{k}$-valued point using  \Cref{lem:sigmaconjI}, we may assume that $m$ is trivial modulo $J$. Since $\phi_x^i(m)$ is $J^{q^i}$-continuous, the product $m\phi_x(m)\phi_x^2(m)\cdots$ converges to a $J$-continuous $(\wtd{A}^b)^\wedge$-valued point $m'$ of $I_M$. The $\sigma$-conjugation by $xm'x^{-1}$ translates $xm$ to $x$. 
\end{proof}

Take $h\in LG((\wtd{A}^b)^\wedge)$ as in \Cref{prop:secondchoice} and let $h^{-1}t^{-1}\sigma(h)=x\bar{n}$. By translating by a $\ov{k}$-valued point using \Cref{lem:sigmaconjII}, we may assume that the reduction of $\bar{n}$ is trivial. Since the class $[h^{-1}]\in \GR_\cl{G}$ is bounded, there is a positive integer $n$ such that for every $y\in \sigma^{-1}(\phi_x^{-n}(I_{\ov{N}}))$ and $k\in L^+\cl{G}$, $(hk)y(hk)^{-1}$ lies in $L^+\cl{G}$. 
By replacing $h$ with $h\sigma^{-1}(\bar{n}^{-1}\phi_x^{-1}(\bar{n}^{-1})\cdots\phi_x^{-n+1}(\bar{n}^{-1}))$, we may assume that $\bar{n}\in \phi_x^{-n}(I_{\ov{N}})$. Note that we can take an infinite product $h\sigma^{-1}(\bar{n}^{-1}\phi_x^{-1}(\bar{n}^{-1})\cdots)$ in $LG((\wtd{A}^b)^\wedge)$ but the product is not continuous any more. This is why we need the contribution of $I^\wedge_{x,n}$. 

Let $I_{x,n}$ be the quotient $\phi_x^{-n}(I_{\ov{N}})/\phi_x^{-n-1}(I_{\ov{N}})$, which is isomorphic to an affine space of dimension $\langle 2\rho, \nu \rangle$. 
This is known in the literature such as \cite{GHKR10} and \cite{HV12} in the split case, and a similar proof can be applied by modifying the $x$-conjugation by $\phi_x$. 
Let $I^\wedge_{x,n}$ be the completion of $I_{x,n}$ at the identity and fix a section of $\phi_x^{-n}(I_\ov{N})\to I_{x,n}$. It gives a continuous section of $I_{x,n}^\wedge$ on $\phi_x^{-n}(I_\ov{N})$. 

\begin{prop}\label{prop:thirdchoice}
There is a continuous section $\bar{n}' \in \phi_x^{-n}(I_{\ov{N}})((\wtd{A}^b)^\wedge)$ reducing to the identity such that $\phi_x^{-1}(\bar{n}')\bar{n}\bar{n}'^{-1}\in I^\wedge_{x,n}$. 
\end{prop}
\begin{proof}
    This is the content of \cite[Theorem 6.5, Step 6]{HV12} in equal characteristic. Here, we explain the argument to indicate that it works well with continuity. 

    Let $J$ be an ideal of definition of $(\wtd{A}^b)^\wedge$ such that $\bar{n}$ is trivial modulo $J$. We construct two sequences $\{\bar{n}_i\}_{i\geq 0}$ in $I^\wedge_{x,n}((\wtd{A}^b)^\wedge)$ and $\{\bar{n}'_i\}_{i\geq 0}$ in $\phi_x^{-n}(I_{\ov{N}})((\wtd{A}^b)^\wedge)$ all of which are continuous and reduce to the identity such that the element $\bar{n}''_i=\phi_x^{-1}(\bar{n}'_i)\bar{n}\bar{n}'^{-1}_i$ is congruent to $\bar{n}_i$ modulo $J^{q^i}$. We set $\bar{n}_0=\bar{n}'_0=1$. Suppose that we have $\bar{n}_i$ and $\bar{n}'_i$. The image of $\bar{n}''_i$ under the quotient map $\phi_x^{-n}(I_{\ov{N}})\to I_{x,n}$ lies in $I_{x,n}^\wedge$, so we may set $\bar{n}_{i+1}$ as the image of $\bar{n}''_i$. Then, $\bar{n}_{i+1}\bar{n}''^{-1}_i$ lies in $\phi_x^{-n-1}(I_{\ov{N}})$, so we may set $\bar{n}'_{i+1}$ as $\phi_x(\bar{n}_{i+1}\bar{n}''^{-1}_i)\bar{n}'_i$, and we have $\bar{n}''_{i+1}=\bar{n}_{i+1}\phi_x(\bar{n}_{i+1}\bar{n}''^{-1}_i)^{-1}$. Since $\bar{n}''_i$ is congruent to $\bar{n}_i$ modulo $J^{q^i}$, $\bar{n}_{i+1}$ is congruent to $\bar{n}_i$ modulo $J^{q^i}$, and thus congruent to $\bar{n}''_i$ modulo $J^{q^i}$. This implies that $\bar{n}''_{i+1}$ is congruent to $\bar{n}_{i+1}$ modulo $J^{q^{i+1}}$. By construction, the sequence $\{\bar{n}'_i\}_{i\geq 0}$ converges to a $J$-continuous section $\bar{n}' \in \phi_x^{-n}(I_{\ov{N}})((\wtd{A}^b)^\wedge)$. It is a desired element since $\phi_x^{-1}(\bar{n}')\bar{n}(\bar{n}')^{-1}$ is equal to the limit of $\{\bar{n}_i\}_{i\geq 0}$. 
\end{proof}

Take $\bar{n}'$ as in \Cref{prop:thirdchoice}. By replacing $h$ with $h\sigma^{-1}(\bar{n}'^{-1})$, we may assume that $\bar{n}\in I_{x,n}^\wedge$. If we set $(\wtd{N}^b)^\wedge = \Spf(\wtd{A}^b)^\wedge$, the element $\bar{n}$ defines a morphism $(\wtd{N}^b)^\wedge\to I^\wedge_{x,n}$.
Let $g'$ be an element of $G(L)$ such that the reduction of $h$ is equal to $g^{-1}g'$. Thanks to the consideration made in the definition of a positive integer $n$, it follows that the section $[g'h^{-1}]\in \GR_\cl{G}((\wtd{A}^b)^\wedge)$ defines a morphism $(\wtd{N}^b)^\wedge \to X_{\leq \mu}(b)^\wedge_{[g]}$. 
This is the content of \cite[Theorem 6.5, Step 5]{HV12} in equal characteristic. 

Now, we have completed the construction of a correspondence describing a local foliation. The key property of this correspondence is the following adicness. 

\begin{thm}\label{thm:constmap}
The product map $(\wtd{N}^b)^\wedge \to X_{\leq \mu}(b)^{\wedge}_{[g]} \mathrel{\widehat{\times}}I^\wedge_{x,n}$ is adic.
\end{thm}
\begin{proof}
    Let $(\wtd{N}^b)^\wedge_0$ be the fiber of the closed point of $X_{\leq \mu}(b)^{\wedge}_{[g]} \mathrel{\widehat{\times}}I^\wedge_{x,n}$ in $(\wtd{N}^b)^\wedge$ as defined right after \Cref{lem:adiccrit}. By \Cref{lem:adiccrit}, it is enough to show that $(\wtd{N}^b)^\wedge_0$ is the closed point of $(\wtd{N}^b)^\wedge$. 

    Let us write $(\wtd{N}^b)^\wedge_0=\Spf(A')$. We see that $[g'h^{-1}_{A'}]=[g]$ and $\bar{n}_{A'}=1$. In particular, if we set $h'=h_{A'}g'^{-1}g$, we have $h'\in L^+\cl{G}(A')$ and $\sigma(h')^{-1}t_{A'}h'=b'^{-1}$. Let $J$ be the continuity of $h'$. Since we have $[\sigma(h')^{-1}t_{A'}]=[b'^{-1}]$, $[t_{A'}]$ is $J^q$-continuous. By the construction using \Cref{thm:constdef}, $t_{A'}$ is also $J^q$-continuous, so $h'=t_{A'}^{-1}\sigma(h')b'^{-1}$ is also $J^q$-continuous. However, since $J$ is the continuity of $h'$, we get $J=J^q$. Since $J$ is topologically nilpotent and $A'$ is separated, it follows that $J=0$. In particular, $h=1$ and $t_{A'}=b'^{-1}$. Then, the composition of the natural maps $(\wtd{N}^b)^\wedge_0\to N^b\to \Spec(A)\to \GR_\cl{G}$ factors through the closed point $[b'^{-1}]$. These maps are all adic in a sense, so $(\wtd{N}^b)^\wedge_0$ is a closed point. 
\end{proof}

This property gives a lower bound of the dimension of $X_{\leq \mu}(b)^{\wedge}_{[g]}$ by \Cref{lem:adicdim} and the bound turns out to be equal to the dimension of $X_{\leq \mu}(b)$. 

\begin{cor}\label{cor:dim}
$\dim X_{\leq \mu}(b)^{\wedge}_{[g]}=\langle\rho,\mu-\nu\rangle-\deff(b)/2$. 
\end{cor}
\begin{proof}
    Thanks to \Cref{lem:adicdim}, \Cref{thm:constmap} implies that $\dim (\wtd{N}^b)^\wedge \leq \dim X_{\leq \mu}(b)^{\wedge}_{[g]}+\dim I_{x,n}^\wedge$. Since dimensions are preserved under perfection and completion as in \Cref{lem:invdimperf} and \Cref{lem:invdimcomp}, we have $\dim (\wtd{N}^b)^\wedge = \dim \wtd{N}^b$. By construction in \Cref{prop:firstchoice} and the estimate in \Cref{prop:dimlwb}, we have $\dim \wtd{N}^b=\dim N^b \geq \langle\rho,\mu+\nu\rangle-\deff(b)/2$. Since $I_{x,n}$ is an affine space of dimension $\langle 2\rho, \nu \rangle$, we obtain the inequality $\dim X_{\leq \mu}(b)^{\wedge}_{[g]}\geq \langle\rho,\mu-\nu\rangle-\deff(b)/2$. The dimension formula implies that the right-hand side is equal to the dimension of $X_{\leq \mu}(b)$, thus we obtain the claim. 
\end{proof}

Now, we arrive at the mixed characteristic counterpart of \cite[Corollary 6.8 (a)]{HV12}. 

\begin{thm}\label{thm:mainthm}
The closed affine Deligne-Lusztig variety $X_{\leq \mu}(b)$ is equidimensional, and the affine Deligne-Lusztig variety $X_\mu(b)$ is a dense open subscheme of $X_{\leq \mu}(b)$. 
In particular, the affine Deligne-Lusztig variety $X_\mu(b)$ is equidimensional. 
\end{thm}
\begin{proof}
The dimension of $X_{\leq \mu}(b)$ at each closed point is constant, so $X_{\leq \mu}(b)$ is equidimensional. Since the complement of $X_\mu(b)$ has smaller dimension than $X_{\leq \mu}(b)$, $X_\mu(b)$ is dense in $X_{\leq \mu}(b)$. 
\end{proof}

\begin{rmk}
    There is a possibility of extending the above result to the case where $G$ is ramified but quasi-split over $F$. In that case, we take a parahoric group scheme $\cl{G}$ so that $\cl{G}(O_L)$ is a special maximal parahoric subgroup of $G(L)$ (see \cite[Lemma 6.1]{Zhu15}). Then, the dimension formula of $X_{\leq \mu}(b)$ is available in \cite[Theorem 2.29]{He16}. It is interesting to see if the same argument works by replacing $P$-fundamental elements with $\sigma$-straight elements (cf.\ \cite[Theorem 1.3]{Nie15}). 
\end{rmk}

\bibliography{reference}
\bibliographystyle{alpha}
\end{document}